\documentclass[10pt,a4paper]{amsart}
\usepackage{hyperref}
\usepackage{mathrsfs}
\usepackage{amsmath}
\usepackage{amssymb}
\usepackage[all]{xy}
\usepackage{latexsym}
\usepackage[T1]{fontenc}
\title{On the Singer functor $R_1$ and the functor $\fix$}
\author[Geoffrey Powell]{Geoffrey M.L. Powell}
\address{Laboratoire Analyse, Géométrie et Applications, UMR 7539\\ Institut
Galilée, Université Paris 13, 93430 Villetaneuse, France}
\email{powell@math.univ-paris13.fr} 
\thanks{This work was partly financed by the project  ANR
BLAN08-2 338236, HGRT.
\newline
\indent
The author would like to thank Jean Lannes for observing that the
copresentation of the functor $R_1$ should lead to a conceptual calculation of
$\fix R_1$ and Aurélien Djament for some helpful comments.}
\newtheorem{thm}{Theorem}[section]
\newtheorem{prop}[thm]{Proposition}
\newtheorem{cor}[thm]{Corollary}
\newtheorem{lem}[thm]{Lemma}

\newtheorem{THM}{Theorem}

\theoremstyle{definition}
\newtheorem{defn}[thm]{Definition}
\newtheorem{exam}[thm]{Example}

\theoremstyle{remark}
\newtheorem{rem}[thm]{Remark}
\newtheorem{nota}[thm]{Notation}

\newcommand{\usquare}{{\field [u].u^2}}

\newcommand{\rhoprime}{\tilde{\rho}^K}

\newcommand{\indec}{Q}
\newcommand{\rprime}{\rtilde^{K}}
\newcommand{\fix}{\mathrm{Fix}}
\newcommand{\unstproj}{\unst_{\mathrm{Proj}}}
\newcommand{\unstred}{{\unst_{\mathrm{red}}}}

\renewcommand{\phi}{\varphi}
\renewcommand{\epsilon}{\varepsilon}
\newcommand{\symm}{\mathfrak{S}}
\newcommand{\zed}{\mathbb{Z}}
\newcommand{\unst}{\mathscr{U}}
\newcommand{\st}{\mathrm{St}}
\newcommand{\dash}{\hspace{-2pt}-\hspace{-2pt}}
\newcommand{\unstalg}[1][]{{\mathscr{K}_{#1}}}
\newcommand{\rtilde}{\tilde{R}}

\newcommand{\redT}{\overline{T}}
\newcommand{\field}{\mathbb{F}}
\newcommand{\kfield}{\mathbb{K}}
\newcommand{\cala}{\mathcal{A}}

\begin{document}

\begin{abstract}
Lannes' $T$-functor is used to give a construction of the
 Singer functor $R_1$ on the category $\unst$ of unstable modules over the
 Steenrod algebra $\cala$. This leads to a direct proof that the
composite functor $\fix R_1$ is naturally equivalent to the identity. Further
properties of the functors $R_1$ are deduced, especially when applied to
reduced and nilclosed unstable modules.
\end{abstract}
\maketitle

\section{Introduction}
\label{sect:intro}

The purpose of this paper is to make explicit the relation between the 
Singer functor $R_1$  and  Lannes' $T$-functor and to deduce some immediate
consequences. The functor
$R_1 : \unst \rightarrow \field[u]\dash \unst$ takes values in the category
 of $\field [u]$-modules in unstable modules, where $\field$ is taken to
be $\field_2$ (analogous results hold at odd primes). The Singer functor is an
important tool, which arose in the work of Singer on the cohomology of the
Steenrod algebra.

The functor $R_1$ has remarkable properties, in particular it preserves
nilclosed unstable modules. This means that it is amenable to study by
$T$-functor technology, which  provides a first approximation, $\rtilde_1$, to
the
functor $R_1$, where $\rtilde_1 M$ is defined as the kernel of a natural
morphism $\field
[u] \otimes M \rightarrow \overline{\field[u]} \otimes \redT M$. The main
result is summarized by the following, in which $\unstproj$ denotes the
full subcategory of projective unstable modules.

\begin{THM}
\label{THM:R1-rtilde}
 The functor $\rtilde_1: \unst \rightarrow \field [u]\dash \unst$ satisfies the
following properties:
\begin{enumerate}
 \item 
the functor $\rtilde_1$ is left exact;
\item
the functor $\rtilde_1$ coincides with $R_1$ on nilclosed unstable modules;
\item
the functor $R_1$ is naturally equivalent to $\rprime_1$, the left Kan
extension of the functor $\rtilde_1 |_{\unstproj}$, hence 
there is a natural transformation $R_1 \cong \rprime_1 \rightarrow \rtilde_1$.
\end{enumerate}
\end{THM}

The identification of $\rprime_1$ with $R_1$ depends on
 properties of the Singer functor $R_1$; this result is not used
in the proofs of the subsequent results of the paper, which can be interpreted
as results for $R_1$, via Theorem \ref{THM:R1-rtilde}.

The functor $\fix : \field[u] \dash \unst \rightarrow \unst$ is the left
adjoint to the functor $\field [u] \otimes - : \unst \rightarrow \field [u]
\dash \unst$. The above leads to a direct proof of the following result. 

\begin{THM}
\label{THM:fix}
 Let $M$ be an unstable module. There is a natural isomorphism 
\[
 \fix \rprime_1 M \stackrel{\cong}{\rightarrow} M.
\]
\end{THM}

This is then used used to show:

\begin{THM}
\label{THM:reduced}
Let $M$ be a reduced unstable module, then the natural transformation
\[
 \rprime_1 M \rightarrow \rtilde_1 M  
\]
is an isomorphism.
\end{THM}

With this result in hand, all properties of the functor $\rprime_1$ can be
deduced from properties of the functor $\rtilde_1$, which are transparent when
considered modulo nilpotent unstable modules. 

When $M$ is a nilclosed unstable module, the techniques used to prove
Theorem \ref{THM:reduced} provide more information:

\begin{THM}
\label{THM:nilclosed}
 Let $M$ be a nilclosed unstable module. Then 
there is a natural exact sequence in
$\field [u]\dash \unst$:
\begin{eqnarray*}
 0
\rightarrow
\rprime_1 M
\rightarrow 
\field [u] \otimes M 
\rightarrow 
\overline{\field[u]} \otimes \redT M 
\rightarrow C_2 M
\rightarrow 
0
\end{eqnarray*}
such that:
\begin{enumerate}
\item
applying the indecomposables functor $\indec  : \field [u] \dash \unst
\rightarrow \unst$
 induces an exact sequence 
\[
 0
\rightarrow 
\Phi M 
\rightarrow 
M 
\rightarrow 
\Sigma \redT M
\rightarrow 
\Sigma (M: \usquare)
\rightarrow 
0;
\]
\item
applying the functor $\fix  : \field [u] \dash \unst \rightarrow \unst$  induces
an exact sequence 
\[
 0
\rightarrow 
M
\rightarrow 
TM 
\rightarrow 
T \redT M 
\rightarrow 
\redT^2 M
\rightarrow 0.
\]
\end{enumerate}
Moreover, the underlying $\field [u]$-module of $C_2 M$ is free.
\end{THM}

Organization of the paper: Section \ref{sect:recollect} provides
background and introduces notation used in the paper and Section
\ref{sect:review} gives a rapid review of the Singer functor $R_1$. The functors
$\rtilde_1$ and $\rprime_1$ are introduced in Section \ref{sect:Rone}, where 
Theorem \ref{THM:R1-rtilde} is proven.  Theorem \ref{THM:fix} is proved in
Section \ref{sect:fix} and Theorem \ref{THM:reduced} in Section
\ref{sect:reduced}, with an application of the methods being given in Corollary
\ref{cor:reduced-R1-Q1}. Theorem \ref{THM:nilclosed} is proved in  Section
\ref{sect:nilclosed} and the appendix proves some auxiliary results on division
functors.

\section{Background}
\label{sect:recollect}

Throughout this paper, $\field$ is the prime field $\field_2$ of
characteristic two (analogues of the results presented here hold in odd
characteristic). The
category of unstable modules over the Steenrod algebra is denoted by $\unst$
and the category of unstable algebras by $\unstalg$. See \cite{sch} for
background on unstable modules and algebras.

\begin{nota}
\ 
\begin{enumerate}
\item
Let $\Phi$ denote the degree-doubling functor on graded vector
spaces. (If $X$ is a graded vector space, $\Phi X$ is
concentrated in even degree and $(\Phi X)_{2n}= X_n$.) 
\item
If $M, N$ are unstable modules, a morphism between the underlying
graded vector spaces is denoted by $\xymatrix{ M \ar@{-->} [r]& N}$.
\end{enumerate}
\end{nota}

The functor $\Phi$ restricts to  a functor $\Phi : \unst \rightarrow
\unst$ and there is a natural morphism of unstable modules $Sq_0 : \Phi M
\rightarrow M$.

Lannes' functor $T_V : \unst \rightarrow \unst$, for $V$ an
elementary abelian $2$-group, is the left adjoint to the functor $H^* (V)
\otimes - : \unst \rightarrow \unst$. The association $V \mapsto T_V$ is
covariantly functorial in $V$; namely, a morphism of $\field$-vector spaces
$V\rightarrow W$ induces a natural transformation $T_V \rightarrow T_W$.
 The reduced $T$-functor, $\redT : \unst \rightarrow \unst$, occurs in the
natural splitting $TM \cong  M \oplus \redT M $ given by $M \cong T_0 M
\rightarrow TM = T_\field M \rightarrow M$ induced by the zero $\field$-vector
space.

If $K$ is an unstable algebra, then $K\dash \unst$ denotes the category of
$K$-modules in $\unst$, which is abelian (see \cite[Section 4.4]{lannes}, for
example). Here, $K$ will be the unstable algebra $\field [u]$ with
$u$ of degree one.

The tensor product induces an exact functor $\field [u] \otimes - :
\unst \rightarrow \field [u]\dash \unst$,  which is left adjoint to the
forgetful
functor $\field [u] \dash \unst \rightarrow \unst$. The functor $\field [u]
\otimes -$ admits a left adjoint, the functor $\fix$. (See \cite{lannes} and
\cite{lz2} for  properties of the functor $\fix$.)

These adjunctions are summarized by the diagram 
\[
 \xymatrix{
\field[u]\dash \unst 
\ar@<3ex>[rr]|(.6){\fix}_(.4)\bot
\ar@{<-}[rr]|(.6){\field [u] \otimes -}_(.4)\bot
\ar@<-3ex>[rr]|(.6){\mathrm{forget}}
&&
\unst,
}
\]
and it is formal that there is a natural isomorphism $T(-) \cong \fix
(\field[u] \otimes -)$. 

The augmentation $\epsilon : \field [u] \rightarrow \field$ induces  a section
$\unst \rightarrow \field [u]\dash \unst $ to the forgetful functor. This admits
a left adjoint, the functor of indecomposables, $\indec : \field [u]\dash
\unst \rightarrow \unst$.

 The following general result on graded connected modules will be used without
further comment.

\begin{lem}
\label{lem:u-torsion-free}
Let $\kfield$ be a field and $\kfield[u]$ be the graded polynomial algebra on a
generator of degree one. A graded, connected $\kfield [u]$-module $M$ is
$u$-torsion free if and only if it is a free $\kfield [u]$-module.
\end{lem}

\subsection{The Singer functor $R_1$}
\label{sect:review}

The definition and the
properties of the Singer functor $R_1$ are reviewed in this section;  for
further details, the reader is referred to  Singer \cite{s2,s} 
and the article \cite{lz} of Lannes and Zarati.

\begin{defn} 
For $M$ an unstable module, 
\begin{enumerate}
\item
let $\xymatrix{ \Phi M \ar@{-->}[r]^{\st_1} & \field [u] \otimes M }$  be
defined by 
$\st_1 ( x ):= \sum u^{|x|-i}\otimes Sq^i x $;
\item   
let $R_1 M$ denote the sub $\field [u] $-module of $\field[u]\otimes M $
generated by the image of $\st_1$.
\end{enumerate}
\end{defn}

\begin{prop}
This construction  defines a functor 
 $ 
 R_1 :\unst \rightarrow \field [u]\dash \unst
 $ 
and, by forgetting the $\field[u]$-module structure, a functor $R_1 :
 \unst\rightarrow \unst $ which satisfies the following
 properties:
\begin{enumerate}
\item
$R_1$ is exact and commutes with limits and colimits;
\item
there exists a unique natural transformation
 $\rho_1 : R_1 M \rightarrow     \Phi M$ which makes the following
     diagram commute
\[
 \xymatrix{
R_1 M 
\ar[d]_{\rho_1}
\ar[r]
&
\field [u] \otimes M 
\ar[d]^{\epsilon \otimes 1}
\\
\Phi M
\ar[r]_{Sq_0}
&
M;
}
\]
\item
there is a natural short exact sequence in $\unst$
\[
0 \rightarrow u R_1 M \rightarrow R_1 M
\stackrel{\rho_1}{\rightarrow} \Phi M 
\rightarrow 0;
\] 
\item
$R_1 M $ is a free $\field[u]$-module on a basis $\st_1 (x)$, as $x$
     ranges over a homogeneous basis of $M$;
\item
if $N \subset M$, then $R_1 N = R_1 M \cap (\field [u] \otimes N)$ as
     a submodule of   $\field[u]\otimes M$;
\item
if $M$ is reduced (respectively nilclosed), then $R_1 M$ is reduced
     (resp. nilclosed);
\item
there is a natural isomorphism $
 R_1 (M\otimes N)
\cong
R_1 M \otimes _{\field[u]} R_1 N
$  in $\field[u]\dash\unst $;
\item
if $K \in \unstalg$ is an unstable algebra, then $R_1 K$ has the
     structure of an unstable algebra and belongs to  the under-category $\field[u]\downarrow \unstalg$.
\end{enumerate}
\end{prop}

\section{Building the Singer functor using  $T$}
\label{sect:Rone}

This section introduces the functor $\rtilde_1$, which is a first
approximation to the Singer functor $R_1$. The functors $\rtilde_1$, $R_1$ are
shown to coincide on the full subcategory of nilclosed unstable modules. 
Then Kan extension gives a functor $\rprime_1$,  which is shown 
to be naturally equivalent to $R_1$.

\begin{defn}
\label{def:sigma-tau}
For $M$  an unstable module, let $ \sigma_M , \tau_M$ denote the natural
morphisms of $\field
[u]\dash \unst$:
\begin{enumerate}
\item
$ 
\tau_M :  \field [u] \otimes M 
\rightarrow 
\field [u] \otimes TM
$ induced by the adjunction unit $M \rightarrow \field [u] \otimes TM$ in
$\unst$;
\item
$
\sigma_M : \field [u] \otimes M \rightarrow 
\field [u] \otimes TM
$ 
 given by applying the
functor $\field [u] \otimes -$ to the natural transformation $M \cong
 T_0 M \rightarrow TM$.
\end{enumerate}
\end{defn}

\begin{defn}
\label{def:rtilde}
Let $\rtilde_1 : \unst \rightarrow \field [u]\dash\unst$ 
 be the functor  determined on an unstable module $M$ by 
\begin{eqnarray}
\label{eqn:def_rtilde}
 \rtilde _1 M := \ker \{ \xymatrix{ 
\field [u] \otimes M
 \ar@<.5ex>[r]^{\tau_M}
\ar@<-.5ex>[r]_{\sigma_M}
&
\field [u] \otimes TM
\}.
}
\end{eqnarray}
\end{defn}

\begin{lem}
 \label{lem:reflexive}
The equalizer (\ref{eqn:def_rtilde}) defining $\rtilde_1 M$ is a reflexive
equalizer.
\end{lem}

\begin{proof}
By definition of a reflexive equalizer, it suffices to exhibit a morphism
$\field [u] \otimes TM \rightarrow \field [u] \otimes M$ which is a common
retract to $\sigma_M$ and $\tau_M$; such a retract is provided by the morphism
of $\field [u]$-modules which is induced by the
projection $TM \twoheadrightarrow M$.
\end{proof}

\begin{prop}
\label{prop:fundamental-rtilde}
\ 
\begin{enumerate}
 \item 
There is a natural monomorphism  $
 \rtilde_1 \hookrightarrow \field [u] \otimes - 
$
of functors $\unst \rightarrow \field[u] \dash
\unst$ and 
for $M$ an unstable module, $\rtilde_1 M $ has free underlying
$\field [u] $-module.
\item
The functor $\rtilde_1 $ preserves the classes of reduced (respectively
nilclosed) unstable modules. 
\item
If $N \subset M$ are unstable modules, then $\rtilde_1 N = \rtilde_1 M \cap
(\field [u] \otimes N)$.
\item
The functor $\rtilde_1 : \unst \rightarrow \field [u] \dash \unst$ is left
exact.
\item
The functor $\rtilde_1$ commutes with coproducts.
\item
If $M$ is an unstable module, then the diagram of $\field[u]$-modules 
\[
 \xymatrix{
\rtilde_1 M 
\ar[r]
\ar@{^(->}[d]
&
\rtilde_1 M [u^{-1}]
\ar@{^(->}[d]
\\
\field [u] \otimes M 
\ar@{^(->}[r]
&
\field[u^{\pm 1} ] \otimes M
}
\]
is cartesian.
\item
The functor $\rtilde_1$ commutes with suspension; more generally, if $M, X$ are
unstable modules where $X$ is locally finite, then there is a natural
isomorphism $\rtilde _1 (M \otimes X) \cong \rtilde _1 (M) \otimes X$.
\end{enumerate}
\end{prop}

\begin{proof} \ 
\begin{enumerate}
 \item
The first statement follows directly from the definition; in particular,
$\rtilde_1 M$ is a sub $\field [u]$-module of $\field [u] \otimes M$ 
and is thus $\field [u]$-free.
\item
$\rtilde_1 M$ is a  sub-object of $\field [u] \otimes M$ and the
quotient is a
submodule of $\field [u] \otimes TM$. If $M$ is reduced (respectively
nilclosed) then $\field [u] \otimes M$ and $\field [u] \otimes TM$ are both
reduced (respectively nilclosed). The result follows.
\item
This statement is a formal consequence of the definition of $\rtilde_1$ as a
natural equalizer.
\item
The left exactness of $\rtilde_1$ is a straightforward verification (the
exactness corresponding to the middle term of a short exact sequence is a 
consequence of the previous statement).
\item
The formation of the equalizer diagram commutes with coproducts, since the
$T$-functor is a left adjoint.
\item
The underlying $\field [u]$-module  $\rtilde_1 M$ is the intersection of
$\field[u] \otimes M$
with the equalizer of the diagram of $\field [u^{\pm 1}]$-modules  given by
localizing:
\[
 \field [u^{\pm 1}] \otimes M
\rightrightarrows
\field  [u^{\pm 1}] \otimes TM.
\]
The statement follows by exactness of localization.
\item
The $T$-functor commutes naturally with suspension, which implies the first
statement. A similar argument gives the general result, using the fact
that $T$ commutes with tensor products and $TX \cong X$ if and only if $X$ is
locally finite.
\end{enumerate}
\end{proof}

\begin{rem}
 The functor $\rtilde_1$ is not exact. For example, the surjection $F(1)
\twoheadrightarrow \Sigma \field$ does not yield a surjection under $\rtilde_1$.
\end{rem}

\subsection{Product structures}

The category $\field [u] \dash \unst$ has tensor structure given by
$\otimes_{\field [u]}$. 

\begin{prop}
\label{prop:product-rtilde}
 Let $M,N$ be unstable modules. There is a binatural isomorphism
\[
\mu_{M,N} :  \rtilde_1 M \otimes_{\field [u] } \rtilde_1 N
 \stackrel{\cong}{\rightarrow}
\rtilde_1 (M \otimes N)
\]
which is induced by the isomorphism 
$
 (\field[u] \otimes M) \otimes_{\field [u] } (\field [u] \otimes N)
\cong 
\field [u] \otimes (M \otimes N).
$
\end{prop}

\begin{proof}
 The $T$-functor commutes with tensor products; this implies that there is a
natural commutative diagram
\[
 \xymatrix{
(\field [u] \otimes M) \otimes_{\field [u]} (\field [u] \otimes N)
\ar[r]^{\tau_M \otimes \tau_N}
\ar[d]_\cong
&
(\field [u] \otimes TM) \otimes_{\field [u]} (\field [u] \otimes TN)
\ar[d]^\cong
\\
\field [u] \otimes (M\otimes N)
\ar[r]
_{\tau_{M\otimes N}}
&
\field [u] \otimes T(M\otimes N).
}
\]
Similarly,  $\sigma_{M\otimes N}$
identifies with $\sigma_M \otimes \sigma _N$. 
The result follows from the fact that $\rtilde_1$ is
defined by a reflexive equalizer by Lemma \ref{lem:reflexive}, together with
the formal fact that the tensor product of two reflexive
equalizers in $\field [u]\dash\unst$ is a reflexive equalizer.
\end{proof}

\begin{rem}
Alternatively, this result can be proved by using
the Künneth isomorphism, by passage to the reduced $T$-functor, as in
Lemma \ref{lem:rtilde-redT}.
\end{rem}

\begin{rem}
The product isomorphism $\mu_{M,N}$  induces a 
surjection in
$\unst$
\[
 \rtilde_1 M \otimes \rtilde_1 N 
\twoheadrightarrow 
\rtilde_1 (M \otimes N)
\]
via the  surjection $ \rtilde_1 M \otimes \rtilde_1 N
\twoheadrightarrow  \rtilde_1 M \otimes_{\field [u] } \rtilde_1 N $.
\end{rem}

\begin{cor}
\label{cor:rtilde-unstalg}
The functor $\rtilde_1 $ restricts to a functor $\unstalg \rightarrow
\field[u]\downarrow\unstalg$. Moreover, if $K$ is an unstable algebra, then
$\rtilde_1$ induces a functor
 $
 K\dash \unst
\rightarrow
\rtilde_1 K \dash \unst.
$ 
\end{cor}

\subsection{Relation with the Singer functor $R_1$}

The comparison between $R_1$ and $\rtilde_1$ on nilclosed unstable modules
relies on comparing the functors on the full subcategory of nilclosed
injectives. For $R_1$, the following calculation is due to Lannes and Zarati.

\begin{prop}
 \cite[Section 5.4.7.5]{lz}
\label{prop:R1-injectives}
Let $V$ be an elementary abelian $2$-group. Then
\[
 R_1 H^* (V) \cong H^* (V\oplus \field) ^{G_V}
\]
where $G_V \subset GL(V \oplus \field)$ is the pointwise stabilizer of $V$.
\end{prop}

There is an analogous result for the functor $\rtilde_1$.

\begin{lem}
\label{lem:rtilde-injectives}
 Let $V$ be an elementary abelian $2$-group. Then
\[
 \rtilde_1 H^* (V) \cong H^* (V\oplus \field) ^{G_V}.
\]
\end{lem}

\begin{proof}
Let us identify the morphisms $\sigma, \tau :\field [u] \otimes  H^* (V)
\rightrightarrows \field [u] \otimes T H^* (V)$. The source identifies with $
H^* (V \oplus \field)$ and the image with $\field ^V \otimes H^*
(V\oplus \field) $ and it suffices to identify each component of the 
morphisms $\sigma, \tau$. Namely, for $v \in V$, there is a surjective
evaluation map of Boolean
algebras $\field ^V \rightarrow \field$ which gives the component indexed by
$v$, which
is a morphism of unstable algebras
\[
 H^* (V \oplus \field) \rightarrow H^* (V \oplus \field).
\]

The $v$-component of $\sigma$ is the identity of $H^* (V \oplus
\field)$, independently of $v$.  It is a straightforward calculation to
show that the $v$-component of $\tau$ is induced by the automorphism of $V
\oplus \field$
which is the identity on $V$ and has component $\field \rightarrow V \oplus
\field$ given by $1 \mapsto (v, 1)$. Namely, as $v$ ranges through $V$, these
morphisms
run through the pointwise stabilizer $G_V \subset GL(V \oplus \field) $ of $V$.
The result follows from the definition of $\rtilde_1 H^* (V)$ as an equalizer.
\end{proof}

\begin{thm}
\label{thm:Rone-rtilde}
The functors $R_1$ and $\rtilde_1$ are canonically isomorphic on the full
subcategory of $\unst$ with objects the nilclosed unstable modules.
\end{thm}

\begin{proof}
 The functor $R_1$ is exact, commutes with colimits  and preserves
the class of nilclosed objects; similarly, the functor $\rtilde_1$ is 
left exact, commutes with coproducts and preserves the class of nilclosed
objects.

Proposition \ref{prop:R1-injectives} and Lemma \ref{lem:rtilde-injectives} show
that the functors
coincide on the injective nilclosed unstable modules $H^* (V)$.  This extends
formally to show that the functors $R_1$ and $\rtilde_1$ coincide on all
nilclosed
injective unstable modules, via the classification of the injective unstable
modules \cite{ls} (see \cite[Theorem 3.14.1]{sch}) and the fact that both
functors commute with coproducts.

The functors $R_1$ and $\rtilde_1$ are both equipped with natural inclusions in
$\field [u] \dash \unst$ to $\field [u] \otimes M$. Hence, it follows that the
functors $R_1$ and $\rtilde_1$ coincide on the full subcategory of nilclosed
injective unstable modules. 

If $M$ is a nilclosed unstable module, there is an injective copresentation
\[
 0  \rightarrow M \rightarrow I^0
\rightarrow I^1
\]
with $I^0, I^1$ nilclosed injective unstable modules. The result follows
formally by left exactness of $R_1$ and $\rtilde_1$, by applying the five-lemma.
\end{proof}

\subsection{The Singer functor $R_1$ via left Kan extension}

The Singer functor $R_1$ is exact; hence, 
using Theorem \ref{thm:Rone-rtilde}, the functor $R_1$ can be recovered by using
left Kan extension of the functor $\rtilde_1$. Namely, the projective objects of
$\unst$ are nilclosed and define a full subcategory $\unstproj \subset
\unst$. Hence the functor $\rtilde _1$ coincides on $\unstproj$ with $R_1$, as
defined by Singer.

\begin{defn}
\label{lem:rprime}
 Let $\rprime_1 : \unst \rightarrow \field [u]\dash \unst$ be the left Kan
extension
of the functor $\rtilde_1 | _{\unstproj}$.
\end{defn}

Explicitly, for $M$ an unstable module, the object $\rprime_1 M $ is the
cokernel of $\rtilde_1 P_1 \rightarrow \rtilde_1 P_0$, where $P_1 \rightarrow
P_0 \rightarrow M \rightarrow 0$ is a projective presentation of $M$ in $\unst$.

\begin{lem}
\label{lem:rprime-rtilde}
There is a natural transformation of functors 
 $
 \rprime_1 \rightarrow \rtilde_1
 $
which is an isomorphism on projective unstable modules. 
\end{lem}

\begin{proof}
Formal. 
\end{proof}

\begin{rem}
This result does not
imply {\em a priori} that  $\rprime_1 \rightarrow
\rtilde _1$ is a natural monomorphism.
\end{rem}

\begin{thm}
\label{thm:rprime-R}
The functors $\rprime_1 $ and $R_1$ are canonically isomorphic.
\end{thm}

\begin{proof}
 Formal consequence of Theorem \ref{thm:Rone-rtilde}, Lemma
\ref{lem:rprime-rtilde} and the fact that $R_1$ is
exact.\end{proof}

\begin{rem}
 It is worthwhile stressing that the proof of Theorem \ref{thm:rprime-R} relies
on two fundamental properties of the Singer functor $R_1$: that it is exact and
that $R_1 H^* (V)$ is a nilclosed unstable module, which is identified by 
Proposition \ref{prop:R1-injectives}.
\end{rem}

\section{The functor $R_1$ and $\fix$}
\label{sect:fix}

Recall that there is a natural isomorphism 
 $
 \fix (N) \cong \field \otimes _{T\field [u]} TN,
$ 
where $\field$ is a $T \field [u]$-algebra via the morphism of unstable
algebras $\iota : T \field [u] \rightarrow \field$ adjoint to the identity
morphism of $\field [u]$ (Cf. \cite[Section 
4.4.3]{lannes}).

\begin{lem}
\label{lem:fix-tau}
Let $M$ be an unstable module; under the natural isomorphisms $\fix (\field [u]
\otimes M) \cong TM$ and $\fix (\field [u] \otimes TM) \cong T^2 M$, the
morphisms 
\[
\fix (\sigma_M),  \fix (\tau_M) : \fix (\field [u] \otimes M)  
\rightarrow 
\fix(\field [u] \otimes TM)
\]
identify respectively  with the morphisms $T M \rightarrow T_{\field ^2} M \cong
T^2 M$ induced by the inclusion $i_1 : \field \rightarrow \field ^2$ of the
first factor (respectively by the diagonal $\delta : \field
\rightarrow \field ^2$).
\end{lem}

\begin{proof}
The degree zero part of $T \field [u]$ is the Boolean algebra
$\field^\field$ and projection onto the degree zero part defines a morphism of
unstable algebras $T \field [u] \twoheadrightarrow \field ^\field$. Under the
$T$-functor, the
morphisms $\sigma_M, \tau_M$ give morphisms of $T\field [u] \dash \unst$ and
hence morphisms of  $\field^\field \dash \unst$:
\[
 \field ^\field \otimes TM 
\rightrightarrows
\field ^\field \otimes T^2 M.
\]

These are determined by the morphisms of unstable modules
\[
 TM 
\rightrightarrows
\field ^\field \otimes T^2 M.
\]

The unstable module 
$\field^\field \otimes TM$ has two components, indexed by the elements $w$ of
$\field$. The corresponding components of the morphisms are recovered by
composing with the morphism induced by the respective evaluation maps $\field
^\field \rightarrow \field$.

The component of the morphism $\tau_M$ corresponding to $w \in \field$ is the 
morphism $TM \rightarrow T^2 M$ induced by the
linear map $\field\rightarrow \field ^2$,  $1 \mapsto (1, w)$. The 
identification of $\fix (\tau_M)$ follows by passing to the quotient via
$\field \otimes_{\field ^\field} -$, which  corresponds to the component $w=1$.

The identification of $\fix (\sigma_M)$ is straightforward.
\end{proof}

\begin{prop}
\label{prop:fix-rtilde}
 Let $M$ be an unstable module. The adjoint to the canonical morphism $\rtilde_1
M \rightarrow
\field [u] \otimes M$ is  a natural isomorphism
\[
 \fix \rtilde_1 M \stackrel{\cong}{\rightarrow} M .
\]
\end{prop}

\begin{proof}
 The functor $\fix$ is exact, hence 
\[
 \fix \rtilde_1 M \cong \ker \{ TM \cong T_\field M \rightrightarrows T^2 M
\cong T_{\field^2 M} \},
\]
the equalizer of $\fix (\sigma_M)$ and $\fix (\tau_M)$; these morphisms 
are given by
Lemma \ref{lem:fix-tau}. The result follows by
identifying the equalizer diagram with the  split equalizer
\[
 M \cong T_{\field^0} M 
\rightarrow 
T_\field M 
\rightrightarrows 
T_{\field^2} M,
\]
which is split via the projections $T_\field M \rightarrow T_{\field^0} M$ and
$T_{\field^2} M
\rightarrow T_\field M$ induced respectively by $\field \rightarrow \field ^0$
and the projection $p_2 : \field^2 \rightarrow \field$.
\end{proof}

\begin{thm}
\label{thm:fix-rprime}
Let $M$ be an unstable module. The natural transformation $\rprime_1 M
\rightarrow \rtilde_1 M$ induces an isomorphism 
\[
 \fix \rprime_1 M \stackrel{\cong}{\rightarrow} \fix \rtilde_1 M
\]
and hence an isomorphism 
 $
 \fix \rprime_1 M \stackrel{\cong}{\rightarrow} M $.
\end{thm}

\begin{proof}
 Choose a projective presentation $P_1 \rightarrow P_0 \rightarrow M
\rightarrow 0$ of $M$ in $\unst$. Then the natural transformation 
 $
 \rprime_1 \rightarrow \rtilde_1
 $
induces a commutative diagram in $\unst$:
\[
 \xymatrix{
\fix \rprime_1 P_1
\ar[r]
\ar[d]_\cong 
&
\fix \rprime_1 P_0
\ar[d]_\cong 
\ar[r]
&
\fix \rprime_1 M
\ar[r]
\ar[d]
&
0\\
\fix \rtilde_1 P_1 
\ar[r]
&
\fix \rtilde_1 P_0 
\ar[r]
&
\fix \rtilde_1 M
\ar[r]
&
0
}
\]
in which the top row is exact by the exactness of $\fix$ and the definition of
$\rprime_1$ and the lower row is exact since it is canonically isomorphic to
$P_1 \rightarrow P_0 \rightarrow M \rightarrow 0$, by Proposition
\ref{prop:fix-rtilde}. The result follows from the
five-lemma.
\end{proof}

\begin{rem}
 Theorem  \ref{thm:rprime-R} was known to Jean Lannes, for the functor $R_1$,
using a different argument. The idea that the copresentation of $\rtilde_1$
should lead to a conceptual proof of this fact is due to Lannes.
\end{rem}

\section{The functor $\rprime_1$ on reduced unstable modules}
\label{sect:reduced}

It is a key fact that the functor $\rprime_1$ coincides with $\rtilde_1$ on
reduced unstable
modules. The aim of this section is to provide a proof of this fact {\em
without} appealing to Theorem \ref{thm:rprime-R} and known  properties of $R_1$.
 (The reader happy to work with properties of $\rprime_1 M$ deduced from
properties of $R_1 M$ via Theorem \ref{thm:rprime-R}, will skip most of this
section, passing directly to Proposition \ref{prop:red-free-R1}, the essential
ingredient to the proof of the main result,
Theorem \ref{thm:rprime-reduced-rtilde}.)

The arguments are based on general properties of modules over a graded
polynomial ring, given in the following section.

\subsection{On graded modules over polynomial rings}

Let $\kfield$ be a field and let $\kfield [u]$ be the graded polynomial
$\kfield$-algebra on a generator $u$ of degree one. The canonical augmentation
$\epsilon : \kfield [u] \rightarrow \kfield$ is defined by $u \mapsto 0$ and the
augmentation ideal is written $\overline{\kfield [u]}$.

\begin{lem}
\label{lem:tfree-quotient}
Let $M$ be a graded connected $\kfield$-vector space and $X$ be a graded
$\kfield [u]$-submodule of $\kfield [u] \otimes M$. Suppose that the diagram of
graded  $\kfield [u]$-modules
\[
\xymatrix{
X \ar[r]
\ar[d]
&
\kfield [u] \otimes M
\ar@{^(->}[d]
\\
X [u^{-1}] 
\ar[r]
&
\kfield [u^{\pm 1}] \otimes M
}
\]
is cartesian. Then the quotient module $(\kfield [u] \otimes M)/ X$ is $\kfield
[u]$-free.
\end{lem}

\begin{proof}
It suffices to show that the quotient module is $u$-torsion free. This is
immediate, since $\kfield [u] \otimes M / X$ embeds in $\kfield [u^{\pm 1}]
\otimes M / X[u^{\pm 1}]$ by the hypothesis that the diagram is cartesian.
\end{proof}

\begin{rem}
\ 
\begin{enumerate}
	\item 
	A graded sub $\kfield [u]$-module $X$ of
$\kfield[u] \otimes M $ is free, hence admits a homogeneous space of generators
$W \subset \kfield [u] \otimes M$ such that $X \cong \kfield [u] \otimes W$.
\item
The kernel of $\kfield [u] \otimes M \stackrel{\epsilon
\otimes M} {\rightarrow} M$ is $\overline{\kfield [u]}\otimes M$.
\end{enumerate}
\end{rem}

\begin{lem}
\label{lem:equiv-cond}
Let $M$ be a graded connected $\kfield$-vector space and $X$ be a graded sub
$\kfield[u]$-module of $\kfield [u] \otimes M$. Then the following conditions
are equivalent:
\begin{enumerate}
\item
the diagram of
graded  $\kfield [u]$-modules
\[
\xymatrix{
X \ar[r]
\ar[d]
&
\kfield [u] \otimes M
\ar@{^(->}[d]
\\
X [u^{-1}] 
\ar[r]
&
\kfield [u^{\pm 1}] \otimes M
}
\]
is cartesian;
\item
$X$ admits a graded space of $\kfield[u]$-generators $W \subset \kfield [u]
\otimes M$
such that the composition $ W \hookrightarrow \kfield [u] \otimes M
\stackrel{\epsilon \otimes M}{\rightarrow}
M$ induced by the augmentation $\epsilon$ of $\kfield [u]$ is a monomorphism.
\end{enumerate}
\end{lem}

\begin{proof}
Suppose that the first condition holds, and write $X_M$ for the image of $X$ in
$M$ under the composite $X \hookrightarrow \kfield [u] \otimes M \rightarrow M$.
 Let $W$ be the image of a choice of (graded) section of  $X
\twoheadrightarrow X_M$. Then $W$ generates a sub $\field [u]$-module $X'$ of
$X$ isomorphic to $\kfield [u] \otimes X_M$. Suppose that $X' \subsetneq X$ is a
proper submodule and let $x \in X$ be an element of least degree which does not
lie in $X'$. By definition of $X_M$, 
there exists $x' \in X'$ and $\delta \in \overline{\kfield[u]} \otimes M$ such
 that $x = x' + \delta$. Then $\delta \in X \cap
(\overline{\kfield [u] } \otimes M)$ is divisible by $u$; thus, using the
pullback
hypothesis, $\delta = u \delta'$, for some $\delta' \in X$. The minimality of
the degree of $x$ implies that $\delta' \in X'$, which establishes the required
contradiction.

For the converse, let $X$ be as in the statement and consider $\tilde{X}:=
X[u^{-1}] \cap \field [u] \otimes M$ in $\field [u^{\pm 1}] \otimes M$. Thus,
$\tilde{X}$ satisfies the first hypothesis and $X \hookrightarrow \tilde{X}$
has
$u$-torsion cokernel. Suppose that the inclusion is proper and choose $0  \neq y
\in
\tilde{X} \backslash X$; there exists a minimal positive
integer $t$ such that
$u^t y \in X$. Write $u^t y  =
\sum_{n\geq 0} u^{n} \otimes w_n$ for homogeneous elements $w_n$ of $W$. By
minimality of $t$, $w_0 \neq 0$. The image of $u^t y $ under $\kfield [u]
\otimes M \stackrel{\epsilon \otimes M}{\rightarrow }M$ is $w_0$; this is
non-zero, by the hypothesis on $W$.

However, $ y $ belongs to $\tilde{X} \subset \kfield [u] \otimes M$ by
hypothesis 
 and thus  $u^t y \in \overline{\kfield[u]}\otimes M$, since $t >0$. This
establishes a contradiction.
\end{proof}

\subsection{Projecting to $\Phi$}

It is useful to give an equivalent definition of $\rtilde_1$ using the reduced
$T$-functor. 

\begin{nota}
\label{nota:taubar}
 For $M$ an unstable module, let $\overline{\tau}_M$ be the morphism of $\field
[u] \dash \unst$  induced by the morphism $M\rightarrow
\overline{\field [u]} \otimes \redT M $ of unstable modules adjoint to the
identity of $\redT M$.
\end{nota}

\begin{lem}
\label{lem:rtilde-redT}
 Let $M$ be an unstable module. Then there is a natural isomorphism $\rtilde_1 M
 \cong \ker \overline{\tau}_M$.
\end{lem}

\begin{proof}
 Straightforward.
\end{proof}

There are canonical morphisms $\rtilde_1 M \hookrightarrow \field [u] \otimes M
\stackrel{\epsilon \otimes M}{\rightarrow }M$ in $\unst$. Recall that 
 $\indec : \field [u] \dash \unst \rightarrow \unst$ is the indecomposables
functor. The projective objects of $\unst$ are nilclosed, hence,  for $P$ a
projective, $Sq_0$ induces a monomorphism $\Phi P \hookrightarrow P$.  

\begin{lem}
\label{lem:rho-P}
 Let $P$ be a projective unstable module; then there is a natural (for $P$ in 
$\unstproj$) commutative diagram 
\[
 \xymatrix{
\rtilde_1 P 
\ar[r]
\ar[d]
_{\tilde{\rho}_P} 
&
\field [u] \otimes P 
\ar[d]^{\epsilon \otimes P}
\\
\Phi P
\ar@{^(->}[r]_{Sq_0}
&
P.
}
\]
Moreover, $\tilde{\rho}_P$ factorizes canonically as 
\[
\xymatrix{
\rtilde_1 P \ar@{->>}[r]
\ar[rd]_{\tilde{\rho}_P}
& 
\indec (\rtilde_1 P) 
\ar[d]
\\
&
\Phi P.
}
\]
\end{lem}

\begin{proof}
Let $\alpha_M : \Omega M \rightarrow \redT M$ be the natural transformation of
Definition \ref{def:alpha}. For $M = P$ a projective unstable module, a formal
adjunction argument shows that $\Omega P \hookrightarrow \redT P$ is a
monomorphism.

The morphism $\overline{\field [u]}\twoheadrightarrow \Sigma \field$ induces a
commutative diagram 
\[
\xymatrix{
 \field[u] \otimes P
\ar[r]^{\overline{\tau}_P}
\ar[d]
&
\overline{\field[u]} \otimes \redT P
\ar[d]
\\
P
\ar[r]
&
\Sigma \redT P.
}
\]
The morphism $P \rightarrow \Sigma \redT P $ factorizes canonically as $P
\rightarrow \Sigma \Omega P \hookrightarrow \Sigma \redT P$, where the second
morphism is a monomorphism by the previous discussion and the first  is
the adjunction unit which features in the short exact sequence of unstable
modules
\[ 
 0 \rightarrow \Phi P \rightarrow P \rightarrow \Sigma \Omega P\rightarrow 0.
\]
Hence the commutative diagram and the definition of $\rtilde_1 P $ as the
kernel of $\overline{\tau}_P$ imply  that $\rtilde_1 P \rightarrow P$
factorizes canonically across $\Phi P \hookrightarrow P $, as required.

The final point is formal, from the adjunction defining the indecomposables
functor, $\indec$.
\end{proof}

\begin{rem}
 The result holds for any nilclosed unstable module, by using
 the fact that $\Omega M \rightarrow \redT M $ is a monomorphism if
$M$ is nilclosed (see Proposition \ref{prop:division}).
\end{rem}

The following Lemma provides the calculational input which is required.

\begin{lem}
\label{lem:rhoP-surj}
 Let $P$ be a projective unstable module. Then the natural transformation 
 $
 \tilde{\rho}_P : \rtilde_1 P \twoheadrightarrow \Phi P
$
is surjective. 
\end{lem}

\begin{proof}
Recall that the unstable modules $F(n)$ ($n \geq 0)$ form a set of
projective generators of $\unst$ and that $F(n) \cong \{ F (1) ^{\otimes n}  \}
^{\symm_n}$. It is sufficient to prove that the morphisms $\tilde{\rho}_{F(n)}$
are surjective, for each $n$. The case $n=0$ is immediate. 

 For $P = F(1)$, the surjectivity is a straightforward calculation: $\Phi F(1)$
is a cyclic unstable module generated by the class of degree two and this class
is in the image of $\tilde{\rho}$. Hence there is a linear morphism $\xymatrix{
\Phi
F(1) \ar@{-->}[r]^\phi &\rtilde_1 F(1)}$ which is a section of $\tilde{\rho}$.

Let $n$ be a natural number; then there is a $\symm_n$-equivariant diagram:
\[
 \xymatrix{
\{\rtilde_1 F(1)
\}^{\otimes n}
\ar[r]
&
\rtilde_1 (F(1) 
^{\otimes n})
\ar[d]
&
\rtilde_1 (F(n))
\ar@{_(->}[l]
\ar[d]^{\tilde{\rho}_{F(n)}}
\\
(\Phi F(1) ) ^{\otimes n}
\ar[r]^\cong
\ar@{-->}[u]^{\phi^{\otimes n}}
&
\Phi( F(1)  ^{\otimes n})
&
\Phi F(n)
\ar@{_(->}[l]
}
\]
in which the top row is induced by the exterior product morphism of Proposition
\ref{prop:product-rtilde}, the symmetric group acts by permuting the tensor
factors in the left hand square of the diagram and trivially on the right hand
column.

The left exactness of $\rtilde_1 $ implies that $\rtilde_1 F(n) \cong
\{\rtilde_1 (F(1) ^{\otimes n})\}^{\symm_n}$; similarly, $\Phi F(n) \cong
\{(\Phi
F(1) ) ^{\otimes n}\}^{\symm_n}$.  Hence the restriction of the
composite morphism to the $\symm_n$-invariants of $(\Phi F(1) ) ^{\otimes n}$ 
 yields a linear section to $\tilde{\rho}_{F(n)}$.
\end{proof}

\begin{rem}
\ 
\begin{enumerate}
\item
In the argument above, it suffices to check that the fundamental class of
$F(n)$ is in the image; this does not provide a significant simplification of
the argument.
\item
This can be used to
recover the description of $R_1$ given by Singer. Namely, the above
construction gives rise to the linear morphism $\xymatrix{ 
\Phi M \ar@{-->}[r]^{\st_1}& \field [u]\otimes M}$. 
\end{enumerate}
\end{rem}

\begin{lem}
\label{lem:proj-rtilde-indec}
 Let $P$ be a projective unstable module. Then, as a graded $\field [u]$-module,
$\rtilde_1 P$ is isomorphic to $\field [u] \otimes \Phi P$ and the morphism
$\tilde{\rho}_P$ induces a natural isomorphism
 $
 \indec (\rtilde_1 P)
\cong
\Phi P.
 $
\end{lem}

\begin{proof}
The $\field [u]$-module $\rtilde_1 P $ satisfies the first hypothesis of Lemma
\ref{lem:equiv-cond}, by Proposition \ref{prop:fundamental-rtilde} (6). Hence,
the equivalent condition shows that $\rtilde_1 P $
is free on $\Phi P$, by Lemma \ref{lem:rhoP-surj}, which identifies the image
in $P$. The final statement is clear.
\end{proof}

These results apply to define a morphism $\rhoprime_M: \rprime_1 M \rightarrow
\Phi M$, for an arbitrary unstable module $M$.

\begin{prop}
\label{prop:rho-prime}
Let $M$ be an unstable module. There exists a unique surjective morphism of
unstable modules
\[
 \rhoprime _M : \rprime_1 M \twoheadrightarrow \Phi M
\]
such that, for any projective presentation  $P_1 \rightarrow P_0
\rightarrow M \rightarrow 0$ of $M$, the diagram
 \begin{eqnarray}
  \label{eqn:rho-prime}
\xymatrix{
\rtilde_1 P_1 
\ar@{->>}[d]_{\tilde{\rho}_{P_1}}
\ar[r]
&
\rtilde_1 P_0
\ar[r]
\ar@{->>}[d]_{\tilde{\rho}_{P_0}}
&
\rprime_1 M
\ar[r]
\ar@{->>}[d]^{\rhoprime_M}
&
0
\\
\Phi P_1 
\ar[r]
&
\Phi P_0 
\ar[r]
&
\Phi M 
\ar[r]
&
0
}
\end{eqnarray}
is commutative.

The morphism $\rhoprime_M$ satisfies the following properties:
\begin{enumerate}
 \item
$\rhoprime_M$ defines a natural surjective transformation $\rprime_1
\twoheadrightarrow \Phi$ of functors taking values in  $\unst$;
\item
$\rhoprime_M$ induces an isomorphism 
$
\indec ( \rprime_1 M)  
\cong
\Phi M 
$;
\item
the following diagram in $\unst$ is commutative
\begin{eqnarray}
 \label{diag:rprime-Sq0}
 \xymatrix{
\rprime_1 M
\ar[r]
\ar@{->>}[d]
_{\rhoprime_M}
&
\field [u] \otimes M
\ar[d]^{\epsilon \otimes M}
\\
\Phi M
\ar[r]_{Sq_0}
&
M.
}
\end{eqnarray}
\end{enumerate}
\end{prop}

\begin{proof}
Choose a projective presentation $P_1 \rightarrow P_0 \rightarrow M \rightarrow
0$ of $M$, then the diagram (\ref{eqn:rho-prime}) defines the morphism
$\rhoprime_M$, which is independent of the choice of presentation; surjectivity
is immediate. The unicity of the construction implies that the morphisms
$\rhoprime_M$ form a natural transformation. 

The functor $\indec$ is right exact, hence induces an exact sequence 
\[
\indec (\rtilde_1 P_1) 
\rightarrow 
\indec (\rtilde_1 P_0) 
\rightarrow 
\indec (\rprime _1 M)
\rightarrow 
0.
\]
Now, the morphism  $\indec (\rtilde_1 P_1) \rightarrow \indec (\rtilde_1 P_0)$
identifies with $\Phi P_1 \rightarrow \Phi P_0$, by Lemma
\ref{lem:proj-rtilde-indec}. This implies the second statement.

The final statement follows from the construction of the morphism
$\tilde{\rho}_{P}$, for $P$ a projective unstable module. 
\end{proof}

Recall that there is a natural transformation $\rprime_1 \rightarrow \rtilde_1
$ hence, for any unstable module $M$, a composite morphism $\rprime_1 M
\rightarrow  \rtilde_1 M \hookrightarrow \field [u] \otimes M$.

\begin{prop}
\label{prop:red-free-R1}
Let $M$ be a reduced unstable module. Then 
\begin{enumerate}
\item
the  morphism $\rprime_1 M \rightarrow \field
[u]
\otimes M $ is a monomorphism;
\item
$\rprime_1 M$ is a free $\field [u]$-module on a graded vector subspace
isomorphic to $\Phi M$;
\item
the quotient $\field [u]$-module $(\field [u] \otimes M) / \rprime_1 M $
is
$\field [u]$-free.
\end{enumerate}
\end{prop}

\begin{proof}
Let $X$ denote the image of $\rprime_1 M$ in $\field [u] \otimes M$, so that
$X$ is an object of $\field [u]\dash \unst$. There is a surjection $\Phi M \cong
\indec (\rprime _1 M) \twoheadrightarrow \indec X$ and the diagram
(\ref{diag:rprime-Sq0}) shows that this is an isomorphism, using the injectivity
of $Sq_0$ for $M$.

Hence, by choice of a subspace of $\field[u]$-generators of $\rtilde_1 M$, 
there are surjective morphisms of $\field [u]$-modules
\[
 \field [u] \otimes \Phi M \twoheadrightarrow \rtilde_1 M \twoheadrightarrow X
\cong \field [u] \otimes \Phi M
\]
such that the composite is an isomorphism. It follows that both morphisms are
isomorphisms of $\field [u]$-modules. This completes the proof of the first two
statements.

The final statement follows from Lemmas \ref{lem:tfree-quotient} and 
\ref{lem:equiv-cond}.
\end{proof}

\begin{thm}
\label{thm:rprime-reduced-rtilde}
Let $M$ be a reduced unstable module, then the canonical morphism 
\[
\rprime_1 M
\rightarrow 
\rtilde_1 M 
\]
is an isomorphism. 
\end{thm}

\begin{proof}
Proposition \ref{prop:red-free-R1} implies that $\rprime_1 M \rightarrow
\rtilde_1 M $
is a monomorphism with quotient $\rtilde_1 M/\rprime_1 M $ which is free as an
$\field
[u]$-module. Moreover, Theorem \ref{thm:fix-rprime} implies that $\fix \rprime_1
M
\rightarrow \fix \rtilde_1 M $ is an isomorphism, thus $\fix (\rtilde_1
M/\rprime_1 M
) =0$. 

Recall from \cite[Proposition 0.8]{lz2} that, for $N$ an object of
$\field[u] \dash \unst$, $\fix (N)=0$ if and only if
$N[u^{-1}] =0$. Thus the condition 
$\fix (\rtilde_1 M/\rprime_1 M
) =0$ implies that $\rtilde_1 M /\rprime_1 M =0$, as required.
\end{proof}

\begin{nota}
 Let $\unstred$ denote the full subcategory of reduced unstable modules
in $\unst$. 
\end{nota}

\begin{cor}
\label{cor:reduced-exact}
The functor $\rtilde_1 |_\unstred : \unstred \rightarrow \field [u]\dash
\unst$ is exact.
\end{cor}

\begin{proof}
 The functor $\rtilde_1$ is left exact, hence it suffices to prove that it
preserves surjections between reduced unstable modules; a straightforward
reduction shows that it is sufficient to establish that it preserves
surjections from a projective to a reduced unstable module. This is an
immediate consequence of Theorem \ref{thm:rprime-reduced-rtilde}.
\end{proof}

\subsection{Further consequences}

\begin{defn}
\label{def:Q1}
Let $C_1 : \unst \rightarrow \field [u] \dash \unst$ be the functor defined by
$C_1 M := \mathrm{image} (\overline{\tau}_M) \subset \overline{\field [u] }
\otimes \redT M$.
\end{defn}

\begin{cor}
 \label{cor:reduced-R1-Q1}
Let $M$ be a reduced unstable module. 
There is a natural short exact sequence
in $\field [u] \dash \unst$:
\begin{eqnarray}
 \label{diag:red-ses}
 0 
\rightarrow 
\rprime_1 M 
\rightarrow 
\field [u] \otimes M
\rightarrow 
C_1 M
\rightarrow 
0,
\end{eqnarray}
such that  the following properties are satisfied:
\begin{enumerate}
 \item 
there is a canonical monomorphism $C_1 M \hookrightarrow \overline{\field [u]}
\otimes \redT M
$ in $\field [u] \dash \unst$;
\item
the underlying unstable module of $C_1 M$ is reduced;
\item
$C_1 M$ is free as an $\field [u]$-module;
\item
applying the functor $\indec$ to the short exact
sequence (\ref{diag:red-ses}) induces the short exact sequence 
 \[
0 \rightarrow \Phi M \rightarrow M \rightarrow \Sigma \Omega M \rightarrow 0.  
 \]
\item
$\fix (C_1 M) \cong \redT M$ and the functor $\fix$ applied to the short exact
sequence (\ref{diag:red-ses}) induces the short exact sequence 
\[
 0 \rightarrow M \rightarrow TM \rightarrow \redT M \rightarrow 0.
\]
\end{enumerate}
Moreover, the functor $C_1|_\unstred : \unstred \rightarrow \field [u]\dash
\unst$ is exact.
\end{cor}

\begin{proof}
\ 
\begin{enumerate}
 \item 
Theorem \ref{thm:rprime-reduced-rtilde} implies that $\rprime_1 M$ is naturally
isomorphic to $\rtilde_1 M$, hence $\rprime_1 M$ is the kernel of
$\overline{\tau}_M$ and $C_1 M $ is the  image; the first statement follows.
\item
$C_1 M$ is a sub-object of the unstable module $\overline{\field [u]} \otimes
\redT M$, which is reduced since $M$ is reduced. 
\item
Similarly, using the fact that $\overline{\field [u]} \otimes
\redT M$ is $u$-torsion free.
\item
The functor $\indec$ is right exact and diagram (\ref{diag:rprime-Sq0}) implies
that $\indec (\rtilde_1 M ) \rightarrow \index (\field [u] \otimes M) $ is the
monomorphism $\Phi M \hookrightarrow M$, using the fact that $M$ is reduced.
Hence $\indec (C_1 M) \cong \Sigma \Omega M$. 
\item
The functor $\fix$ is exact and the morphism $\fix (\rtilde_1 M ) \rightarrow
\fix (\field [u] \otimes M) $ identifies with the canonical inclusion $M
\hookrightarrow TM$, by  Proposition \ref{prop:fix-rtilde}. The result follows.
\end{enumerate}
The final statement concerning the exactness of $C_1 |_\unstred$ follows
from the exactness of $\rtilde_1 |_\unstred$ established in Corollary
\ref{cor:reduced-exact} and the exactness of $\redT$.
\end{proof}

\subsection{The functor $\rprime_1$ on nilclosed unstable modules}
\label{sect:nilclosed}

 The considerations of Corollary \ref{cor:reduced-R1-Q1} can be pushed further
if the natural transformation $\alpha_M : \Omega M \rightarrow \redT M $ (see
Definition \ref{def:alpha}) is injective; this is the case  if $M$ is nilclosed,
by
Proposition \ref{prop:division}.

The following definitions apply to arbitrary unstable modules.

\begin{defn}
\label{def:Q2}
Let $C_2 :\unst \rightarrow \field [u]\dash
\unst$ be the functor  $C_2 M:= \mathrm{coker} (\overline{\tau}_M)$.
\end{defn}

The division functor $(- : \usquare) $ is introduced in Section
\ref{sect:division}.

\begin{thm}
\label{thm:R1-Q2-nilclosed}
 Let $M$ be a nilclosed unstable module. Then 
\begin{enumerate}
 \item 
there is a natural exact sequence in
$\field [u]\dash \unst$:
\begin{eqnarray}
 \label{eqn:R2-Q2}
 0
\rightarrow
\rprime_1 M
\rightarrow 
\field [u] \otimes M 
\rightarrow 
\overline{\field[u]} \otimes \redT M 
\rightarrow C_2 M
\rightarrow 
0;
\end{eqnarray}
\item
applying the functor $\indec  : \field [u] \dash \unst \rightarrow \unst$ to
the diagram (\ref{eqn:R2-Q2}) induces an exact sequence 
\[
 0
\rightarrow 
\Phi M 
\rightarrow 
M 
\rightarrow 
\Sigma \redT M
\rightarrow 
\Sigma (M: \usquare)
\rightarrow 
0;
\]
\item
the underlying $\field [u]$-module of $C_2 M$ is free on $\Sigma (M:
\usquare)
$;
\item
applying the functor $\fix  : \field [u] \dash \unst \rightarrow \unst$ to
the diagram (\ref{eqn:R2-Q2}) induces an exact sequence 
\[
 0
\rightarrow 
M
\rightarrow 
TM 
\rightarrow 
T \redT M 
\rightarrow 
\redT^2 M
\rightarrow 0,
\]
in which the morphism $TM \rightarrow T \redT M$ is the reduction of the
morphism $T_\field M \rightarrow T_{\field ^2 } M $ induced by the diagonal map
$\field \rightarrow \field ^2$.
\end{enumerate}
\end{thm}

\begin{proof}
 The first statement is an immediate consequence of the definitions, since
$\rprime_1 M \cong \rtilde_1 M$ when $M$ is reduced.

The
morphism
$\overline{\tau}_M$ factorizes (by definition of $C_1 M$) as
\[
 \field [u] \otimes M 
\twoheadrightarrow C_1 M
\stackrel{}{\hookrightarrow}
\overline{\field [u] } \otimes \redT M.
\]
By Corollary \ref{cor:reduced-R1-Q1}, applying the
functor $\indec$ to these
morphisms induces 
\[
 M \twoheadrightarrow \Sigma \Omega M 
\stackrel{\Sigma \alpha_M}{\hookrightarrow} 
\Sigma \redT M 
\]
where the monomorphism follows from the hypothesis that $M$ is nilclosed, by
Proposition \ref{prop:division}. This establishes the second
statement, by the right exactness of the functor $\indec$, using Proposition
\ref{prop:division} to identify the cokernel.

The underlying $\field [u]$-module of $\overline{\field [u]} \otimes \redT M$
is free on $\Sigma \redT M$. The fact that $C_2 M$ is $\field [u]$-free  follows
from Lemmas \ref{lem:tfree-quotient} and  \ref{lem:equiv-cond}, since the
indecomposables of $C_1 M $ map injectively to the indecomposables of
$\overline{\field [u]} \otimes \redT M$.

The functor $\fix$ is exact; moreover the morphism $\overline{\tau}_M$
factorizes naturally in $\field [u]\dash \unst$ as 
\[
 \xymatrix{
 \field [u] \otimes M 
\ar[r]^{\tau_M}
& 
\field [u] \otimes TM 
\ar[r]
&
\overline{\field [u]} \otimes \redT M .
}
\]
This gives the identification of the morphism $\fix(\overline{\tau}_M)$, by
Lemma \ref{lem:fix-tau}. This morphism factorizes as 
\[
 TM \twoheadrightarrow \redT M 
\hookrightarrow 
T \redT M \cong \redT M \oplus \redT ^2 M,
\]
where the second morphism is the product of the identity on $\redT M$ and the
reduced diagonal $\redT M \rightarrow \redT ^2 M$. The cokernel is natural
isomorphic to $\redT^2 M$. 
\end{proof}

\appendix
\section{Results on division functors}
\label{sect:appendix}

The aim of this appendix is to provide a proof of Proposition
\ref{prop:division}.

\begin{nota}
 For $M$ an unstable module, let $(-: M)$ denote the division functor which is
left adjoint to $M \otimes - : \unst \rightarrow \unst$ and let $(-:M)_n$
denote its $n$th left derived functor. 
\end{nota}

\subsection{The loop functor on nilclosed unstable modules}
An unstable module $M$ is reduced if and only
if $Sq_0$ induces a
short exact
sequence 
 $
 0
\rightarrow 
\Phi M 
\stackrel{Sq_0}{\rightarrow}
M 
\rightarrow 
\Sigma \Omega M
\rightarrow 
0, 
 $
where $\Omega$ is the division functor $( - : \Sigma \field)$. This is
equivalent to the condition that $\Omega_1 M=0$.

\begin{lem}
\label{lem:loop-tensor}
 Let $M,N$ be reduced unstable modules. There exists a natural short exact
sequence 
\[
 0
\rightarrow 
\Omega (M \otimes N) 
\rightarrow 
(\Omega M \otimes N) \oplus (M \otimes \Omega N) 
\rightarrow 
\Sigma (\Omega M \otimes \Omega N)
\rightarrow 
0.
\]

\end{lem}

\begin{proof}
 The short exact sequence is given by the Künneth theorem applied to the tensor
product of the complexes $(M \rightarrow \Sigma \Omega M)$ and $(N \rightarrow
\Sigma \Omega N)$.
\end{proof}

\begin{lem}
\label{lem:omega-tensor-reduced}
 Let $M,N$ be unstable modules such that $M, N , \Omega M, \Omega N$ are
reduced. Then $\Omega (M \otimes N)$ is reduced.
\end{lem}

\begin{proof}
 The unstable module $\Omega (M \otimes N)$ embeds in $(\Omega M \otimes N)
\oplus (M \otimes \Omega N)$, which is reduced, by the hypothesis.
\end{proof}

\begin{lem}
\label{lem:omega-reduced-HV}
 The unstable module $ \Omega H^* (V)$ is reduced, for $V$ a finitely-generated
elementary abelian $2$-group. 
\end{lem}

\begin{proof}
 The unstable modules $H^* (\zed/2) \cong \field [u]$ and $\Omega H^* (\zed/2)
\cong \Phi \field [u]$ are reduced. The result follows by 
induction, using Lemma \ref{lem:omega-tensor-reduced}.
\end{proof}

\begin{prop}
\label{prop:omega-nilclosed}
If  $M$ is a nilclosed unstable module, then $\Omega M$
is reduced. 
\end{prop}

\begin{proof}
 The functor $\Omega$ is a left adjoint, hence commutes with colimits.
Moreover, the colimit of a diagram of reduced unstable modules is reduced. 
A nilclosed unstable module is the filtered colimit of its finitely-generated
nilclosed sub-modules, hence it is sufficient to prove the result in the case
that $M$ is finitely-generated.

Suppose that $M$ is a finitely-generated nilclosed unstable module. There
exists a finitely-generated elementary abelian $2$-group $V$ and  a short exact
sequence 
\[
 0 \rightarrow M \rightarrow H^* (V) \rightarrow Q \rightarrow 0
\]
of unstable modules, with $Q$ reduced, since $M$ is nilclosed. The unstable
module $\Omega_1 Q$ is trivial, since $Q$ is reduced, hence the functor $\Omega$
induces a monomorphism $\Omega M \rightarrow \Omega H^* (V)$. By Lemma
\ref{lem:omega-reduced-HV}, $\Omega H^* (V)$ is reduced, hence the result
follows.
\end{proof}

\subsection{The division functor $(-: \usquare)$ and the natural
transformation $\alpha$}
\label{sect:division}

\begin{defn}
\label{def:alpha}
 Let $\alpha : \Omega = (- : \Sigma \field) \rightarrow \redT = (- :
\overline{\field [u]}) $ be the natural transformation 
 induced on division functors by  the unique non-trivial morphism
$p : \overline{\field
[u]} \rightarrow \Sigma \field$. 
\end{defn}

\begin{nota}
Let $\usquare$ denote the kernel of $
\overline{\field[u]}
\stackrel{p}{\rightarrow} 
\Sigma \field,
$
which identifies as an object of $\field[u]\dash \unst$ with the ideal
of $\field[u]$ generated by $u^2$. 
\end{nota}

\begin{prop}
\label{prop:division}
 The division functor $(-: \usquare)$ is right exact and the left
derived functors $(-: \usquare)_n$ are trivial for $n >2$.
Moreover, there is a natural exact sequence 
\[
 0 
\rightarrow 
(M: \usquare)_1
\rightarrow 
\Omega M 
\stackrel{\alpha_M}{\rightarrow}
\redT M
\rightarrow
(M: \usquare)
\rightarrow 
0
\]
and a natural isomorphism $(-: \usquare)_2 \cong \Omega_1 M$.
In particular,  if $M$ is reduced, then $(M: \usquare)_2=0$.

If $M$ is nilclosed, then $(M: \usquare)_n=0$ for $n >0$,
hence $\alpha_M$ is a monomorphism.
\end{prop}

\begin{proof}
 The first part of the statement follows from formal considerations,
using the exactness of $\redT$ and the vanishing of the left derived
functors $(- : \Sigma \field)_n $ for $n>2$. 

It remains to consider the case where $M$ is nilclosed. Proposition
\ref{prop:omega-nilclosed} implies that $\Omega M$ is reduced, hence the
subobject $(M:\usquare) _1$ is reduced. Thus, it is sufficient to
show that $T_V \big( (M:\usquare) _1 \big)$ is trivial in degree
zero, for each finitely-generated elementary abelian $2$-group $V$ and for every
nilclosed unstable module $M$. The functor $T_V$ commutes  with $(
-:\usquare) _1$ and preserves nilclosed unstable modules, hence it
is sufficient to show that $(M:\usquare) _1$ is trivial in degree
zero for every nilclosed unstable module $M$. Equivalently, it is sufficient to
show that $\alpha_M : \Omega M \rightarrow \redT M$ is a monomorphism in degree
zero. 

There is a natural isomorphism $(\Omega M) ^0 \cong M^1$.  Since $M$ is
nilclosed, the inclusion of the elements of
degree one induces a natural monomorphism $F(1)  \otimes M^1
\hookrightarrow M$ of nilclosed unstable modules,  where $M^1$ is
considered as an unstable module concentrated in degree zero. The
induced morphism $\Omega (F(1) \otimes M^1 ) \rightarrow \Omega M$ is a
monomorphism (since the cokernel of $F(1)  \otimes M^1
\hookrightarrow M$ is reduced), which is an isomorphism in degree
zero. 

Hence, by naturality, it suffices to prove that $\alpha_{F(1) \otimes M^1}$ is a
monomorphism. This follows formally since $F(1) \otimes M^1$ is projective.
(Alternatively, a direct calculation gives the result.)
\end{proof}

\begin{exam}
\label{exam:alpha-notinj}
The condition that $M$ be nilclosed in the final statement cannot be weakened to
reduced. Consider $M= \Phi F(1) $; the
unstable module $\Omega M$ is isomorphic to $\Sigma \field$ and the
morphism $\alpha_{\Phi F(1)} : \Omega \Phi F(1) \cong \Sigma \field \rightarrow
\redT \Phi F(1) \cong \field$ is
trivial.
\end{exam}

\begin{exam}
\label{exam:nilclosed-exact}
The functor $(-: \usquare)$ is not exact
when restricted to $\unstred$. There is a short exact sequence in $\unstred$:
\[
 0 
\rightarrow 
\Lambda^2 (F(1))
\rightarrow 
F(2) 
\rightarrow 
\Phi F(1)
\rightarrow
0
\]
in which the two nonzero left hand terms are nilclosed. By Example
\ref{exam:alpha-notinj}, $(\Phi F(1): \usquare)_1= \Sigma
\field$, hence  Proposition \ref{prop:division} shows that the division functor
$(-: \usquare)$ does not send this to a short exact sequence.
\end{exam}

\providecommand{\bysame}{\leavevmode\hbox to3em{\hrulefill}\thinspace}
\providecommand{\MR}{\relax\ifhmode\unskip\space\fi MR }
\providecommand{\MRhref}[2]{%
  \href{http://www.ams.org/mathscinet-getitem?mr=#1}{#2}
}
\providecommand{\href}[2]{#2}

\end{document}